\documentclass[a4paper]{amsart}
\usepackage[utf8]{inputenc}
\usepackage[british]{babel}
\usepackage{csquotes}
\usepackage{amsmath}
\usepackage{amsfonts}
\usepackage{amsthm}
\usepackage{amssymb}
\usepackage{mathrsfs}
\usepackage{mathtools}
\usepackage{breakurl}
\usepackage[hidelinks, pdfauthor={Johannes Schmitt}, pdftitle={The class group of a minimal model of a quotient singularity}]{hyperref}
\usepackage[style=alphabetic, maxbibnames=10, maxalphanames=4]{biblatex}
\usepackage[useregional]{datetime2}
\usepackage[shortcuts]{extdash}
\usepackage{tikz}
\usepackage{tikz-cd}

\renewcommand{\phi}{\varphi}
\renewcommand{\theta}{\vartheta}

\theoremstyle{plain}
\newtheorem{theorem}{Theorem}[section]
\newtheorem{lemma}[theorem]{Lemma}
\newtheorem{proposition}[theorem]{Proposition}
\newtheorem{corollary}[theorem]{Corollary}
\theoremstyle{definition}
\newtheorem{definition}[theorem]{Definition}
\theoremstyle{remark}
\newtheorem{remark}[theorem]{Remark}
\newtheorem{notation}[theorem]{Notation}
\newtheorem{example}[theorem]{Example}

\DeclareMathOperator{\Ab}{Ab}
\DeclareMathOperator{\age}{age}
\DeclareMathOperator{\Cl}{Cl}
\DeclareMathOperator{\coker}{coker}
\DeclareMathOperator{\diag}{diag}
\let\div\relax
\DeclareMathOperator{\div}{div}
\DeclareMathOperator{\Div}{Div}
\DeclareMathOperator{\GL}{GL}
\DeclareMathOperator{\Hom}{Hom}
\DeclareMathOperator{\id}{id}
\DeclareMathOperator{\rk}{rk}
\DeclareMathOperator{\SL}{SL}
\DeclareMathOperator{\Spec}{Spec}

\newcommand{\C}{\mathbb C}
\newcommand{\Q}{\mathbb Q}
\newcommand{\R}{\mathcal R}
\newcommand{\Z}{\mathbb Z}
\newcommand{\free}{\mathrm{free}}
\newcommand{\tors}{\mathrm{tors}}

\makeatletter
\patchcmd{\@setauthors}{\footnotesize}{}{}{}
\patchcmd{\@setauthors}{\MakeUppercase}{}{}{}
\patchcmd{\@setaddresses}{\scshape}{}{}{}
\patchcmd{\contentsnamefont}{\scshape}{\bfseries}{}{}
\patchcmd{\abstract}{\scshape}{\bfseries}{}{}
\patchcmd{\section}{\scshape}{\bfseries}{}{}
\patchcmd{\@secnumfont}{\mdseries}{\bfseries}{}{}
\patchcmd{\@captionheadfont}{\scshape}{\bfseries}{}{}
\makeatother

\title[The class group of a minimal model of a quotient singularity]{The class group of a minimal model\\ of a quotient singularity}

\author{Johannes Schmitt}
\address{Johannes Schmitt, Naturwissenschaftlich-Technische Fakultät, Universität Siegen, Walter-Flex-Straße 3, 57068 Siegen, Germany}
\email{johannes2.schmitt@uni-siegen.de}
\date{\DTMdisplaydate{2024}{03}{15}{-1}}

\begin{document}

\begin{abstract}
  Let \(V\) be a finite\-/dimensional vector space over the complex numbers and let \(G\leq \SL(V)\) be a finite group.
  We describe the class group of a minimal model (that is, \(\Q\)\=/factorial terminalization) of the linear quotient \(V/G\).
  We prove that such a class group is completely controlled by the junior elements contained in \(G\).
\end{abstract}

\maketitle
\tableofcontents

\section{Introduction}
\label{sec:intro}

Let \(V\) be a finite\-/dimensional vector space over \(\C\) and let \(G\leq\GL(V)\) be a finite group.
By the classical theorem of Chevalley--Serre--Shephard--Todd \cite[Théorème~1']{Ser68}, the \emph{linear quotient} \[V/G \coloneqq \Spec \C[V]^G\] of \(V\) by \(G\) is smooth if and only if \(G\) is generated by \emph{reflections}, that is, elements \(g\in G\) with \(\rk (g - \id_V) = 1\).
If \(V/G\) is singular, we also refer to this affine variety as a \emph{quotient singularity}.
Regarding the class group of Weil divisors \(\Cl(V/G)\) of \(V/G\), there is a nowadays well\-/known theorem by Benson \cite{Ben93} which says that \(\Cl(V/G) \cong \Hom(G/K, \C^\times)\), where \(K\leq G\) is the subgroup generated by the reflections contained in \(G\).
Restricting to subgroups \(G\leq\SL(V)\), we will show that the class group of a certain partial resolution of \(V/G\) can be described in a similar way.

The inclusion \(G\leq\SL(V)\) implies that \(G\) cannot contain any reflections and \(V/G\) is a singular variety.
Further, by \cite{BCHM10}, there is a \emph{\(\Q\)\=/factorial terminalization} -- or \emph{minimal model} -- \(\phi:X\to V/G\), that is, a crepant, partial resolution of singularities of \(V/G\) with at most terminal singularities; see below for the precise definition.
\(\Q\)\=/factorial terminalizations of quotient singularities are well studied and, starting with the famous \emph{McKay correspondence} \cite{McK80}, it is often observed that properties of \(\phi:X\to V/G\) are controlled only by the action of \(G\) on \(V\), see for example the survey \cite{Rei02}.
This leads to the expectation that questions regarding the birational geometry of \(V/G\) should be answered by only considering \(G\), see Reid's `principle of the McKay correspondence' \cite[Principle~1.1]{Rei02}.

In this note, we give further evidence of this phenomenon by describing the class group \(\Cl(X)\) of \(X\).
By a version of the McKay correspondence due to Ito and Reid \cite{IR96}, the rank of the free part of this finitely generated abelian group coincides with the number of conjugacy classes of \emph{junior elements} in \(G\); we recall the definition of these distinguished elements of \(G\) below.
In the literature, one further finds a sufficient condition and a full characterization of the freeness of \(\Cl(X)\) by Donten\-/Bury--Wiśniewski \cite[Lemma~2.11]{DW17} and Yamagishi \cite[Proposition~4.14]{Yam18}, respectively.

We study the torsion part \(\Cl(X)^\tors\) of \(\Cl(X)\) and obtain a theorem, which reads similar to Benson's result on \(\Cl(V/G)\).
\begin{theorem}[{= Theorem \ref{thm:clgrp}}]
  \label{thm:intro}
  Let \(G\leq \SL(V)\) be a finite group and let \(H\leq G\) be the subgroup generated by the junior elements contained in \(G\).
  Let \(\phi: X\to V/G\) be a \(\Q\)\=/factorial terminalization of \(V/G\).
  Then we have a canonical isomorphism of abelian groups \[\Cl(X)^\tors \cong \Hom(G/H, \C^\times)\;,\] which is induced by the push\-/forward map \(\phi_\ast: \Cl(X)\to \Cl(V/G)\).
\end{theorem}

Combining our result with \cite{IR96} gives a complete description of the class group.
\begin{corollary}[{= Corollary \ref{cor:clgrpfull}}]
  \label{cor:intro}
  With the assumptions in Theorem~\ref{thm:intro}, we have \[\Cl(X) \cong \Z^m \oplus \Hom(G/H, \C^\times)\;,\] where \(m\) is the number of conjugacy classes of junior elements in \(G\).
\end{corollary}

\begin{remark}
  We emphasize that by Corollary~\ref{cor:intro}, the class group of a \(\Q\)\=/factorial terminalization is completely controlled by the group \(G\) itself.
  This agrees well with the mentioned principle of Reid \cite[Principle~1.1]{Rei02}.
\end{remark}

\begin{remark}
  We make a further philosophical observation.
  Recall that by the theorem of Chevalley--Serre--Shephard--Todd, the variety \(V/G\) is smooth if and only if \(G\leq \GL(V)\) is generated by reflections.
  This is mirrored by a theorem of Yamagishi \cite[Theorem~1.1]{Yam18} which generalizes a result of Verbitsky \cite[Theorem~1.1]{Ver00} and says that, if a \(\Q\)\=/factorial terminalization \(X\to V/G\) is smooth, then \(G\leq \SL(V)\) must be generated by junior elements.

  We feel that Theorem~\ref{thm:intro} mirrors Benson's theorem (Theorem~\ref{thm:clgrpVG}) in the same way.
  In both cases the geometry of the linear quotient \(V/G\) is controlled by the reflections contained in \(G\) and the junior elements control the geometry of the \(\Q\)\=/factorial terminalization \(X\to V/G\).
  Still, it appears that this picture is far from complete.
  The theorem of Verbitsky and Yamagishi on the smoothness of \(X\) is not an equivalence and the freeness of the class group also depends in a somewhat convoluted way on the junior elements, see Remark~\ref{rem:clgrp}.
\end{remark}

Notice that the mentioned theorem of Verbitsky and Yamagishi together with our result implies that if \(X\) is smooth, then the class group is free.
This result was already shown in \cite[Lemma~2.11]{DW17} using different methods; a stronger statement appears in \cite[Proposition~4.14]{Yam18}, see also Corollary~\ref{cor:clgrp}.

To prove Theorem~\ref{thm:intro}, we study the Cox ring \(\R(V/G)\) which is graded by \(\Cl(V/G)\).
This directly extends the work of \cite{Yam18}.
The Cox ring \(\R(V/G)\) was introduced by Cox \cite{Cox95} in the context of toric varieties and transferred to a more general setting of birational geometry by Hu and Keel \cite{HK00}.

After fixing the notation and recalling fundamental results in Section~\ref{sec:pre}, we present a correspondence of homogeneous elements in \(\R(V/G)\) and effective divisors on \(V/G\) in Section~\ref{sec:corr}.
We do not claim originality for the results in these sections, but only unify the different sources and transfer them to our setting.
In Section~\ref{sec:grad}, we analyse the grading of the Cox ring \(\R(V/G)\) by \(\Cl(V/G)\) in more detail and finally use this in Section~\ref{sec:clgrp} to prove our theorem.
We close with some small examples in Section~\ref{sec:ex}.

\subsection*{Acknowledgements}
This paper originates from my PhD thesis \cite{Sch23}; I thank my supervisor Ulrich Thiel as well as Alastair Craw for valuable comments and for encouraging me to publish this result separately.
I thank the anonymous referees for helpful comments that improved the presentation in this article and David J.\ Benson for pointing out \cite[Theorem~2.11]{Nak82} in the context of Theorem~\ref{thm:clgrpVG} to me.
This work was supported by the SFB\-/TRR~195 `Symbolic Tools in Mathematics and their Application' of the German Research Foundation (DFG).

\section{Preliminaries}
\label{sec:pre}

Let \(V\) be a finite\-/dimensional vector space over \(\C\) and let \(G\leq \GL(V)\) be a finite group.
Note that we will restrict to subgroups of \(\SL(V)\) after a short general discussion.
Write \(V/G \coloneqq \Spec \C[V]^G\) for the linear quotient, where \(\C[V]^G\) is the invariant ring of \(G\).
The variety \(V/G\) is normal and \(\Q\)\=/factorial, see \cite{Ben93}.

\subsection{The class group of \texorpdfstring{\(V/G\)}{V/G}}
Recall that by a \emph{reflection} (or \emph{pseudo\-/reflection}), we mean an element \(g\in\GL(V)\) with \(\rk(g - \id_V) = 1\).
The class group of \(V/G\) was first described by Benson \cite[Theorem~3.9.2]{Ben93} building on work of Nakajima \cite[Theorem~2.11]{Nak82}.
\begin{theorem}[Benson]
  \label{thm:clgrpVG}
  Let \(G\leq \GL(V)\) be a finite group and let \(K\leq G\) be the subgroup generated by the reflections contained in \(G\).
  Then there is an isomorphism \[\Cl(V/G)\cong \Hom(G/K, \C^\times)\] of abelian groups.
\end{theorem}

With \(K\) as in the theorem, let \(\Ab(G/K)\coloneqq (G/K)/[G/K, G/K]\) be the abelianization of \(G/K\) and write \(\Ab(G/K)^\vee\) for the group of irreducible (hence linear) characters of this group.
By elementary character theory \cite[Theorem~I.9.5]{BKZ18}, we have \[\Hom(G/K,\C^\times) = \Ab(G/K)^\vee\] and we hence often write \(\Ab(G/K)^\vee\) for the class group of \(V/G\).

\subsection{\texorpdfstring{\(\Q\)\=/factorial}{Q-factorial} terminalizations}
From now on, let \(G\leq \SL(V)\) be a finite subgroup.

\begin{definition}[\(\Q\)\=/factorial terminalization]
  Let \(Y\) be a normal \(\Q\)\=/factorial variety.
  A \emph{\(\Q\)\=/factorial terminalization} of \(Y\) is a projective birational morphism \(\phi: X\to Y\) such that \(X\) is a normal \(\Q\)\=/factorial variety with terminal singularities and \(\phi\) is crepant.
\end{definition}

In our context, a \(\Q\)\=/factorial terminalization \(X\to V/G\) is often referred to as \emph{minimal model}, see for example \cite{IR96}.
However, the usage of this terminology is not entirely uniform, which is why we decided to use the more unwieldy term `\(\Q\)\=/factorial terminalization'.

We have the following special case of the deep result achieved in \cite{BCHM10}.
\begin{theorem}[Birkar--Cascini--Hacon--McKernan]
  Let \(G\leq \SL(V)\) be a finite group.
  There exists a \(\Q\)\=/factorial terminalization of \(V/G\).
\end{theorem}
\noindent This is \cite[Corollary~1.4.3]{BCHM10} together with the fact that \(V/G\) has canonical singularities by the Reid--Tai criterion \cite[Theorem~3.21]{Kol13}, see \cite[Theorem~2.1.15]{Sch23} for details.

\subsection{McKay correspondence}
Throughout let \(\phi: X\to V/G\) be a \(\Q\)\=/factorial terminalization.
A deep connection between \(X\to V/G\) and the group \(G\) itself is given by a version of the McKay correspondence due to Ito and Reid \cite{IR96}.
We recall some notation from \cite{IR96}.

For the following definition, let \(g\in \GL(V)\) be of finite order \(r\) and fix a primitive \(r\)\=/th root of unity \(\zeta_r\).
Then there are integers \(0\leq a_i < r\), such that the eigenvalues of \(g\) are given by the powers \(\zeta_r^{a_1},\dots, \zeta_r^{a_n}\), where \(\dim V = n\).
\begin{definition}[Age and junior elements]
  We call the number \[\age(g)\coloneqq \frac{1}{r}\sum_{i = 1}^na_i\] the \emph{age} of \(g\).
  Elements of age 1 are called \emph{junior}.
\end{definition}

By construction, the number \(\age(g)\) is an integer if \(g\in \SL(V)\) and the junior elements in \(\SL(V)\) are hence the non\-/trivial elements of minimal age 1.
The age is by definition invariant under conjugacy and we refer to the conjugacy classes of a group \(G\leq \SL(V)\) consisting (only) of junior elements as \emph{junior conjugacy classes}.

\begin{remark}
  \label{rem:agewelldef}
  The age as defined above depends on the choice of the root of unity \(\zeta_r\), although this is hidden in the notation.
  See \cite[p.\ 224, Remark~3]{IR96} for an example demonstrating this.
  However, the subgroup generated by the junior elements, which is relevant for Theorem~\ref{thm:intro}, is independent of any choices.
  We give a short argument for this in Appendix~\ref{sec:app}, see Lemma~\ref{lem:app2}.
\end{remark}

\begin{definition}[Monomial valuation]
  For non\-/negative integers \(a_1,\dots,a_n\in \Z_{\geq 0}\) with \(\gcd(a_1,\dots, a_n) = 1\), we construct a discrete valuation \(v : \C(x_1,\dots, x_n) \to \Z\) defined on \(\C[x_1,\dots, x_n]\) via \[\sum_{\alpha\in\Z_{\geq 0}^n}\lambda_\alpha x_1^{\alpha_1}\cdots x_n^{\alpha_n} \mapsto \min_{\substack{\alpha\in \Z_{\geq 0}^n\\ \lambda_\alpha\neq 0}}\sum_{i = 1}^n\alpha_ia_i\;.\]
  We call \(v\) a \emph{monomial valuation.}
\end{definition}
\noindent This construction indeed gives a well\-/defined discrete valuation, see \cite[Definition~2.1]{Kal02}.

\begin{notation}
  Let \(g\in\SL(V)\) of finite order \(r\) be a junior element with respect to the primitive \(r\)\=/th root of unity \(\zeta_r\).
  That is, the eigenvalues of \(g\) are \(\zeta_r^{a_1},\dots, \zeta_r^{a_n}\) with \(\frac{1}{r}\sum_{i = 1}^na_i = 1\).
  For \(d \coloneqq \gcd(a_1,\dots,a_n)\), we obtain \(\frac{1}{d}\) as the age of \(g\) with respect to the root of unity \(\zeta_r^d\).
  As \(\age(g)\) is an integer, we conclude \(d = 1\) and we can therefore define a monomial valuation \[v_g:\C(x_1,\dots, x_n)\to \Z\] for \(g\) via \(a_1,\dots,a_n\).
\end{notation}

The construction of \(v_g\) again depends on the choice of a root of unity, see Remark~\ref{rem:agewelldef} above.
The valuation \(v_g\) is stable under conjugacy of \(g\) and we can hence associate valuations to conjugacy classes in \(G\) without needing to specify a particular representative.

\begin{theorem}[Ito--Reid, `McKay correspondence']
  \label{thm:mckay}
  Let \(G\leq\SL(V)\) be a finite group and let \(X\to V/G\) be a \(\Q\)\=/factorial terminalization.
  Then there is a one\-/to\-/one correspondence between the junior conjugacy classes of \(G\) and the irreducible exceptional divisors on \(X\).

  More precisely, if \(E\) is a divisor corresponding to a conjugacy class of a junior element \(g\in G\) of order \(r\) in this way, then \(v_E = \frac{1}{r}v_g\), where \(v_E\) is the valuation of \(E\) and we identify \(\C(X) = \C(V)^G\) via the birational morphism \(X\to V/G\).
\end{theorem}
\noindent See \cite[Section~2.8]{IR96} for a proof.

Let \(m\in\Z_{\geq 0}\) be the number of junior conjugacy classes in \(G\).
Using \cite[Proposition~II.6.5~(c)]{Har77}, we have an exact sequence of abelian groups
\begin{align}
  \begin{tikzcd}[ampersand replacement=\&] \bigoplus_{i = 1}^m \Z E_i \arrow{r} \& \Cl(X) \arrow["\phi_\ast"]{r} \& \Cl(V/G) \arrow{r} \& 0\;,\end{tikzcd}\label{eq:ses1}
\end{align} where \(E_i\) are the irreducible exceptional divisors on \(X\).
As \(\Cl(V/G)\) is finitely generated, this implies that \(\Cl(X)\) is finitely generated as well.

\begin{notation}
  We write \(\Cl(X)^\tors\leq \Cl(X)\) for the torsion subgroup of \(\Cl(X)\) and \(\Cl(X)^\free\) for the corresponding factor group, that is, \(\Cl(X)^\free = \Cl(X)/\Cl(X)^\tors\).
\end{notation}

The sequence in \eqref{eq:ses1} can be extended to a short exact sequence by Grab \cite{Gra19}.
\begin{proposition}[Grab]
  \label{prop:ses}
  There is a short exact sequence of abelian groups \[\begin{tikzcd}0\arrow{r} & \bigoplus_{i = 1}^m \Z E_i \arrow{r} & \Cl(X)\arrow["\phi_\ast"]{r} & \Cl(V/G) \arrow{r} & 0\;,\end{tikzcd}\] where \(E_i\in\Div(X)\) are the irreducible components of the exceptional divisor of \(\phi\) and \(\phi_\ast : \Cl(X)\to \Cl(V/G)\) is the induced push\-/forward map.
\end{proposition}
\noindent See \cite[Proposition~4.1.3]{Gra19} for a proof.

\begin{remark}
  \label{rem:free}
  As \(\Cl(V/G)\) is a torsion group, we can deduce with Theorem~\ref{thm:mckay} and Proposition~\ref{prop:ses} that \(\Cl(X)^\free \cong \Z^m\) for the free part of \(\Cl(X)\), where \(m\in\Z_{\geq 0}\) is the number of junior conjugacy classes in \(G\).
\end{remark}

\subsection{Cox rings}
To understand the torsion part \(\Cl(X)^\tors\) of \(\Cl(X)\) we use the \emph{Cox ring} \(\R(V/G)\) of \(V/G\).
The precise definition of this ring for a normal variety \(Y\) is a bit involved and here we just remark that if the class group \(\Cl(Y)\) is a free group, we have \[\R(Y) = \bigoplus_{[D]\in \Cl(Y)}\Gamma(Y, \mathcal O_Y(D))\;,\] see \cite[Section~1.4]{ADHL15} for the general case.
The Cox ring is well\-/defined for normal varieties \(Y\) with finitely generated class group \(\Cl(Y)\) under the additional assumption \(\Gamma(Y, \mathcal O_Y^\times) = \C^\times\).
These properties are in particular fulfilled for linear quotients \(V/G\) and we may hence speak about the Cox ring \(\R(V/G)\).

We summarize a result on the Cox ring of \(\R(V/G)\) by Arzhantsev and Ga\v{\i}fullin \cite{AG10}.
Recall that \(\Cl(V/G) \cong \Ab(G)^\vee\) as \(G\leq \SL(V)\) cannot contain any reflections.
There is an action of \(\Ab(G)\) on the ring \(\C[V]^{[G,G]}\) induced by the action of \(G\).
This action induces a grading by \(\Ab(G)^\vee\) by setting the graded component of a character \(\chi\in\Ab(G)^\vee\) to be \[\C[V]^{[G,G]}_\chi \coloneqq \{f\in\C[V]^{[G,G]}\mid \gamma.f = \chi(\gamma) f\text{ for all }\gamma\in\Ab(G)\}\;,\] where we write \(\gamma.f\) for the action of \(\gamma\in\Ab(G)\) on \(f\in \C[V]^{[G,G]}\).
\begin{theorem}[Arzhantsev--Ga\v{\i}fullin]
  \label{thm:coxVG}
  Let \(G\leq \SL(V)\) be a finite group.
  Then there is an \(\Ab(G)^\vee\)\=/graded isomorphism \(\R(V/G) \cong \C[V]^{[G,G]}\).
\end{theorem}
\noindent See \cite[Theorem~3.1]{AG10} for a proof.

\section{Correspondence of effective divisors and homogeneous elements}
\label{sec:corr}

To be able to deduce information about \(\Cl(X)\), we use the connection between effective divisors and \emph{canonical sections} in the Cox ring \(\R(V/G)\) of \(V/G\).
We recall this correspondence and adapt it to our setting.

\begin{notation}
  For a divisor \(D\in \Div(V/G)\), we write \(\chi_{[D]}\in\Ab(G)^\vee\) for the character corresponding to the class \([D]\in\Cl(V/G)\) under the isomorphism in Theorem~\ref{thm:clgrpVG}.
\end{notation}

\begin{remark}
  \label{rem:isos}
  Working with the ring \(\R(V/G)\) brings two subtle problems.
  First of all, homogeneous elements \(f\in\R(V/G)\) are only residue classes of elements of the function field \(\C(V)^G\) as \(\Cl(V/G)\) is a torsion group.
  We hence cannot immediately identify such elements \(f\) with a function in \(\C(V)^G\).
  However, for a divisor \(D\in\Div(V/G)\) we have an isomorphism \[\psi_D : \Gamma(V/G, \mathcal O_{V/G}(D)) \to \R(V/G)_{[D]}\;\] by \cite[Lemma~1.4.3.4]{ADHL15}.
  That means, once we fixed a representative of the degree of a homogeneous element \(f\in \R(V/G)\) we can uniquely lift \(f\) to an element of \(\C(V)^G\).

  The second problem comes from the fact that we make heavy use of the graded isomorphism \(\Psi : \R(V/G) \to \C[V]^{[G,G]}\) as in Theorem~\ref{thm:coxVG} to the extent that one might forget that the isomorphism is not an identity.
  This is in particular important when we work with a valuation \(v:\C(V)\to\Z\).
  We can only use \(v\) on elements of \(\C[V]^{[G,G]}\) and cannot apply \(v\) to elements of \(\R(V/G)\) in a well\-/defined way without choosing a system of representatives for the class group.
  For \(D\in \Div(V/G)\), we have an isomorphism of vector spaces \[\tilde\psi_D : \Gamma(V/G,\mathcal O_{V/G}(D)) \to \C[V]^{[G,G]}_{\chi_{[D]}}\] by setting \(\tilde \psi_D \coloneqq\Psi\circ\psi_D\).
  Notice that for the trivial divisor, this gives an identity as we have \[\Gamma(V/G, \mathcal O_{V/G}(0)) = \C[V]^G = \C[V]^{[G,G]}_1\;,\] where 1 denotes the trivial character.
\end{remark}

\begin{notation}
  Let \(\chi\in\Ab(G)^\vee\) and let \(D\in\Div(V/G)\) with \(\chi = \chi_{[D]}\).
  For a homogeneous element \(0\neq f\in \C[V]^{[G,G]}_\chi\), let \(\tilde f\in \C(V)^G\) be the rational function mapping to \(f\) via the isomorphism determined by \(D\) as in Remark \ref{rem:isos}.
  We associate to \(f\) an effective divisor \[\div_{[D]}(f)\coloneqq \div(\tilde f) + D\in\Div(V/G)\;,\] the \textit{\([D]\)\=/divisor} of \(f\).
  This construction is well\-/defined, see \cite[Proposition~1.5.2.2]{ADHL15}.
  In particular, the \([D]\)\=/divisor is independent of the choice of the representative \(D\).
  We have \([{\div_{[D]}(f)}] = [D]\) by definition.
\end{notation}

The construction of a \([D]\)\=/divisor is not limited to our setting; see \cite[Construction~1.5.2.1]{ADHL15} for more details and the general case.
We point out that \(f\in\C[V]^{[G,G]}\) is in general not an element of \(\C(V)^G\), that is, there is no meaning in writing \(\div(f)\).

The \([D]\)\=/divisor behaves well with respect to the multiplication of elements.
\begin{lemma}
  \label{lem:Ddivadd}
  For non\-/zero homogeneous elements \(f\in \C[V]^{[G,G]}_{\chi_{[D_1]}}\) and \(g\in \C[V]^{[G,G]}_{\chi_{[D_2]}}\), we have \[\div_{[D_1] + [D_2]}(fg) = \div_{[D_1]}(f) + \div_{[D_2]}(g)\;.\]
\end{lemma}
\noindent See \cite[Proposition~1.5.2.2~(iii)]{ADHL15} for a proof.

We have a converse to the construction of the \([D]\)\=/divisor.
\begin{proposition}
  \label{prop:Ddivcorr}
  Let \(E\in\Div(V/G)\) be an effective divisor.
  There exist a class \([D]\in\Cl(V/G)\) and an element \(f\in \C[V]^{[G,G]}_{\chi_{[D]}}\) with \(E = \div_{[D]}(f)\).
  The element \(f\) is unique up to constants; it is called a \emph{canonical section} of \(E\).
\end{proposition}
\noindent See \cite[Proposition~1.5.2.2~(i)]{ADHL15} and \cite[Proposition~1.5.3.5~(ii)]{ADHL15} for a proof.

Using the correspondence between effective divisors and homogeneous elements one can derive a precise description of the image of the strict transform of an effective divisor \(D\in\Div(V/G)\) in the free group \(\Cl(X)^\mathrm{free}\).
The general idea of this argument appeared to our knowledge first in \cite[Lemma~3.22]{DW17}.
We require a bit of notation.

Recall that by Theorem~\ref{thm:mckay} we have a one\-/to\-/one correspondence between the junior conjugacy classes of \(G\) and the irreducible components of the exceptional divisor of \(\phi : X \to V/G\).
Let \(\{g_1,\dots, g_m\}\in G\) be a minimal set of representatives of the junior conjugacy classes corresponding to exceptional prime divisors \(E_1,\dots, E_m\in\Div(X)\).
For each \(i\in\{1,\dots, m\}\), write \(v_i\) for the monomial valuation on \(\C(V)\) defined by \(g_i\) and recall from Theorem~\ref{thm:mckay} that we have \(v_{E_i} = \frac{1}{r_i}v_i\), where \(v_{E_i}\) is the divisorial valuation of \(E_i\) and \(r_i\) the order of \(g_i\).

The following also appears in \cite[Proposition~4.1.9]{Gra19}.
We present the argument from \cite[Lemma~4.3]{Yam18} for completeness.
Denote the canonical projection by \(\rho : \Cl(X)\to \Cl(X)^\mathrm{free}\).
\begin{proposition}
  \label{prop:divfree}
  Let \(D\geq 0\) be an effective divisor on \(V/G\) and let \(f\in \C[V]^{[G,G]}_{\chi_{[D]}}\) be a canonical section.
  Write \(\overline D\coloneqq \phi_\ast^{-1}(D)\) for the strict transform of \(D\) via \(\phi\).
  Then we have the equality \[\rho([\overline D]) = -\sum_{i = 1}^m \frac{1}{r_i}v_i(f) \rho([E_i])\] in \(\Cl(X)^\mathrm{free}\).
\end{proposition}
\begin{proof}
  As \(f\) is homogeneous with respect to the action of \(\Ab(G)\), there is \(r \in\Z_{> 0}\) such that \(f^r\in \C[V]^{[G,G]}_1 = \C[V]^G\subseteq \C(V)^G\) and \(rD\) is principal.
  In particular, we have \[rD = \div_{[rD]}(f^r) = \div_{[0]}(f^r) = \div(f^r)\;,\] where the first equality is by Lemma~\ref{lem:Ddivadd}, the second by the independence of choice of representative and the third is by the fact that \(f^r\in\C[V]^G\), see Remark~\ref{rem:isos}.
  Then we have \[\div(\phi^\ast(f^r)) = r\overline D + \sum_{i = 1}^mv_{E_i}(\phi^\ast(f^r)) E_i\;.\]
  Hence, we have the equality of classes \[[r\overline D] = -\sum_{i = 1}^m v_{E_i}(\phi^\ast(f^r)) [E_i]\] in \(\Cl(X)\).
  Now \(v_{E_i}(\phi^\ast(f^r)) = \frac{1}{r_i}v_i(f^r)\) by Theorem~\ref{thm:mckay}.
  Noting that \(v_i\) is a valuation on \(\C(V)\) (and not just \(\C(V)^G\)) this yields \[[r\overline D] = -\sum_{i = 1}^m \frac{r}{r_i}v_i(f) [E_i]\;.\]
  We may finally cancel \(r\) in the free group \(\Cl(X)^\mathrm{free}\) giving \[\rho([\overline D]) = - \sum_{i = 1}^m \frac{1}{r_i}v_i(f)\rho([E_i])\;.\]
  \vskip-\baselineskip\qedhere
\end{proof}

\begin{remark}
  Notice that the proof of Proposition~\ref{prop:divfree} in fact computes the degree of a preimage of the canonical section \(f\) under a graded surjective morphism \(\R(X) \to \R(V/G)\) induced by \(\phi\).
  See \cite[Proposition~4.1.3.1]{ADHL15} for the construction of this morphism between the Cox rings.
\end{remark}

\section{A digression on gradings}
\label{sec:grad}

To understand the group \(\Cl(X)\), we first have to get a better understanding of the grading of \(\C[V]^{[G,G]}\) by \(\Ab(G)^\vee\).
Unfortunately, there are a few subtle details involved, turning this into quite a technical discussion.

Again, let \(g_1,\dots, g_m\in G\) be representatives of the junior conjugacy classes corresponding to the exceptional divisors \(E_1,\dots, E_m\in \Div(X)\) of \(\phi\) and write \(v_1,\dots,v_m\) for the monomial valuations corresponding to the \(g_i\).

At first, fix \(i\in \{1,\dots, m\}\).
Let the eigenvalues of \(g_i\) be given by \(\zeta_{r_i}^{a_{i, 1}}, \dots, \zeta_{r_i}^{a_{i, n}}\) with a primitive \(r_i\)\=/th root of unity \(\zeta_{r_i}\) and integers \(0\leq a_{i, j} < r_i\), where \(r_i\) is the order of \(g_i\) in \(G\) and \(n = \dim V\).
This induces a \(\Z\)\=/grading \(\deg_i\) on \(\C[x_1,\dots, x_n]\) by putting \(\deg_i(x_j) \coloneqq a_{i, j}\).
For a polynomial \(f\in \C[x_1,\dots, x_n]\), the valuation \(v_i(f)\) is then the degree of the homogeneous component of \(f\) of minimal degree with respect to \(\deg_i\).
Notice that in this construction, we consider \(V\) in an eigenbasis of \(g_i\) giving rise to the isomorphism \(\C[V] \cong \C[x_1,\dots, x_n]\).
However, the grading \(\deg_i\) is well\-/defined on \(\C[V]\) for any basis of \(V\), although the variables of the polynomial ring are in general not homogeneous.
As we endow the same ring with gradings by different groups, we use the non\-/standard notation \((\C[V],\Z,\deg_i)\) for the ring \(\C[V]\) graded by \(\Z\) via \(\deg_i\).

The group \(\langle g_i\rangle\) acts on \(\C[V]\) and hence induces a grading by \(\langle g_i\rangle^\vee \cong \Z/r_i\Z\), which we denote by \(\overline{\deg}_i\).
Write \((\C[V], \Z/r_i\Z, \overline{\deg}_i)\) for the ring \(\C[V]\) graded by \(\Z/r_i\Z\) via \(\overline{\deg}_i\).
We directly obtain:

\begin{lemma}
  \label{lem:grad1}
  With the above notation, if \(f\in \C[V]\) is \(\deg_i\)\=/homogeneous, then \(f\) is \(\overline{\deg}_i\)\=/homogeneous as well and we have \[\deg_i(f)\equiv \overline{\deg}_i(f) \mod r_i\;.\]
  In particular, there is a graded morphism \[(\C[V], \Z,\deg_i) \to (\C[V], \Z/r_i\Z,\overline{\deg}_i)\] given by the identity on the rings and by the projection \(\Z\to\Z/r_i\Z\) on the grading groups.
\end{lemma}

Write \(g_i.f\) for the linear action of \(g_i\) on \(f\in \C[V]\).
Observe that for every \(1\leq i \leq m\) we have an action of \(g_i\) on \(\C[V]^{[G,G]}\).
Indeed, for any \(f\in \C[V]^{[G,G]}\) and \(h\in [G,G]\), we have \[h.(g_i.f) = (hg_i).f = (hg_i).((g_i^{-1}h^{-1}g_i h).f) = g_i.(h.f) = g_i.f\,,\] so \(g_i.f\in \C[V]^{[G,G]}\) as required.
Hence the grading by \(\langle g_i\rangle^\vee\) descends to \(\C[V]^{[G,G]}\).
As the actions of the elements \(g_1,\dots, g_m\) on \(\C[V]^{[G,G]}\) commute, we can consider all the induced gradings at the same time and hence obtain a grading by the group \(\Z/r_1\Z\times\cdots\times \Z/r_m\Z\) on \(\C[V]^{[G,G]}\).

The \(g_i\) do not commute with each other in general, so we cannot decompose their actions on \(\C[V]\) into a common eigenbasis.
Hence, we \emph{cannot} put the above gradings together to obtain a grading by \(\Z^m\) or \(\Z/r_1\Z\times\cdots\times\Z/r_m\Z\) on \(\C[V]\) as there are in general no polynomials which are homogeneous with respect to all gradings at the same time.

Let \(H\leq G\) be the subgroup of \(G\) generated by the junior elements contained in \(G\).
In general, the representatives \(g_1,\dots, g_m\) do not suffice to generate \(H\).
Let \[\overline H \coloneqq H/(H\cap [G,G])\leq \Ab(G)\] and notice that this group is generated by the residue classes \(\overline g_1,\dots, \overline g_m\) modulo \([G,G]\).
This gives a map \[\langle g_1\rangle\times\cdots\times \langle g_m\rangle \to \Ab(G)\;,\] which is surjective onto \(\overline H\).
This surjection corresponds to an embedding of character groups \(\overline H^\vee\to \Z/r_1\Z\times\cdots\times\Z/r_m\Z\).
Further, the inclusion \(\overline H \to \Ab(G)\) induces a projection of characters \(\Ab(G)^\vee \to \overline H^\vee\) by restriction.
We conclude:

\begin{lemma}
  \label{lem:grad2}
  The gradings on \(\C[V]^{[G,G]}\) coming from the actions of the groups \(\Ab(G)\), \(\overline H\) and \(\langle g_1\rangle\times\cdots\times \langle g_m\rangle\) are compatible in the sense that there is a graded morphism \[(\C[V]^{[G,G]}, \Ab(G)^\vee)\to (\C[V]^{[G,G]}, \Z/r_1\Z\times\cdots\times\Z/r_m\Z)\] which factors through \((\C[V]^{[G,G]}, \overline H^\vee)\).
\end{lemma}

We state for later reference:
\begin{lemma}
  \label{lem:hbar}
  We have \(\Ab(G/H) \cong \Ab(G)/\overline H\) and \(\overline H^\vee \cong \Ab(G)^\vee/\!\Ab(G/H)^\vee\).
\end{lemma}
\begin{proof}
  For the first statement, we note that the image of \([G,G]\) under the projection \(G\to G/H\) is \([G/H, G/H]\).
  Hence, \[\Ab(G/H) \cong (G/H)/([G,G]/[G,G]\cap H) \cong G/(H[G,G])\] and an application of the isomorphism theorems gives the claim.
  The second statement follows directly as \(^\vee\) is a contravariant functor.
\end{proof}

The following three lemmas are key ingredients for our theorem on \(\Cl(X)\).
\begin{lemma}
  \label{lem:clgrp1}
  Let \(f\in \C[V]^{[G,G]}\) be \(\Ab(G)^\vee\)\=/homogeneous.
  For every index \(i\in\{1,\dots, m\}\), we have \(v_i(f) \equiv \overline{\deg}_i(f)\ \operatorname{mod}\ r_i\).
\end{lemma}
\begin{proof}
  Let \(f\in \C[V]^{[G,G]}\) be \(\Ab(G)^\vee\)\=/homogeneous.
  Fix an \(i\in\{1,\dots, m\}\).
  Then \(f\) is \(\overline{\deg}_i\)\=/homogeneous by Lemma~\ref{lem:grad2}.
  By Lemma~\ref{lem:grad1}, there exist \(\deg_i\)- and \(\overline{\deg}_i\)\=/homogeneous elements \(f_{i,j}\in \C[V]\) such that \(f = \sum_j f_{i, j}\) and \(\deg_i(f_{i, j}) < \deg_i(f_{i, j'})\) whenever \(j < j'\).
  In particular, we have \(\deg_i(f_{i, 1}) = v_i(f)\) and \(\overline{\deg}_i(f_{i, 1}) = \overline{\deg}_i(f)\).
  Hence, we conclude \[v_i(f) \equiv \overline{\deg}_i(f_{i, 1}) = \overline{\deg}_i(f)\ \operatorname{mod}\ r_i\] by Lemma~\ref{lem:grad1}.
\end{proof}

\begin{lemma}
  \label{lem:clgrp2}
  Let \(f\in \C[V]^{[G,G]}\) be \(\Ab(G)^\vee\)\=/homogeneous.
  We have \(r_i\mid v_i(f)\) for all \(i\in\{1,\dots, m\}\) if and only if \(f\in \C[V]^H\), where \(H\leq G\) is the subgroup generated by the junior elements contained in \(G\).
\end{lemma}
\begin{proof}
  By Lemma \ref{lem:clgrp1}, we have \(v_i(f) \equiv \overline{\deg}_i(f)\) mod \(r_i\) for every \(i\).
  Therefore, \(r_i \mid v_i(f)\) is equivalent to \(\overline{\deg_i}(f) = 0\) for every \(i\).
  Equivalently, every \(g_i\) acts trivially on \(f\).
  Since \(f\) is furthermore \([G,G]\)\=/invariant, we conclude that this is the case if and only if every junior element in \(G\) leaves \(f\) invariant and hence \(f \in \C[V]^H\).
\end{proof}

\begin{lemma}
  \label{lem:clgrp3}
  Let \([D]\in \Cl(V/G)\) be a class of divisors.
  There exists a homogeneous element in \(\C[V]^{[G,G]}\) of degree \(\chi_{[D]}\).
\end{lemma}
\begin{proof}
  This is saying that the relative invariants with respect to the linear characters of \(\Ab(G)\) on \(\C[V]^{[G,G]}\) are non\-/empty which holds by \cite[Lemma~2.1]{Nak82}.
\end{proof}
Notice that the lemma also implies that we can find an effective divisor in any class of divisors in \(\Cl(V/G)\).

\section{The class group}
\label{sec:clgrp}

We are now prepared for our theorem.
\begin{theorem}
  \label{thm:clgrp}
  Let \(G\leq \SL(V)\) be a finite group and let \(H\leq G\) be the subgroup generated by the junior elements contained in \(G\).
  Let \(\phi: X\to V/G\) be a \(\Q\)\=/factorial terminalization of \(V/G\).
  Then we have a canonical isomorphism of abelian groups \[\Cl(X)^\tors \cong \Ab(G/H)^\vee = \Hom(G/H, \C^\times)\;,\] which is induced by the push\-/forward map \(\phi_\ast: \Cl(X)\to \Cl(V/G)\).
\end{theorem}
\begin{proof}
  For ease of notation, we identify \(\Cl(V/G)\) with \(\Ab(G)^\vee\) via Theorem~\ref{thm:clgrpVG} and use both groups synonymously.
  Notice that \(\Ab(G/H)^\vee\) is the subgroup of \(\Ab(G)^\vee\) consisting of those characters which take value 1 on every junior element.
  We claim that restricting \(\phi_\ast\) to \(\Cl(X)^\tors\) induces a bijection onto \(\Ab(G/H)^\vee\).

  We first show that we indeed have \(\phi_\ast(\Cl(X)^\tors) \subseteq \Ab(G/H)^\vee\).
  Let \(D\in\Div(X)\) be a divisor on \(X\).
  By Lemma~\ref{lem:clgrp3}, there is \(f\in \C[V]^{[G,G]}\) of degree \(\chi_{[\phi_\ast D]}\) and we have the effective divisor \(D' \coloneqq \div_{[\phi_\ast D]}(f)\) on \(V/G\) with \([D'] = [\phi_\ast D]\).
  Write \(\overline{D'}\in\Div(X)\) for the strict transform of \(D'\) via \(\phi\).
  Then \(\phi_\ast \overline{D'} = D'\), hence by Proposition~\ref{prop:ses} we have \begin{align}[\overline{D'}] = [D] + \sum_{i = 1}^m a_i [E_i]\;,\label{eq:clgrp1}\end{align} with \(a_i\in \Z\) and where \(E_1,\dots, E_m\in \Div(X)\) are the irreducible components of the exceptional divisor of \(\phi\).
  As before let \(\rho:\Cl(X) \to \Cl(X)^\mathrm{free}\coloneqq \Cl(X)/\Cl(X)^\tors\) be the canonical projection.
  Applying \(\rho\) on both sides of \eqref{eq:clgrp1} and using Proposition~\ref{prop:divfree} yields \begin{align}\rho([D]) = -\sum_{i = 1}^m\frac{1}{r_i}v_i(f)\rho([E_i]) - \sum_{i = 1}^ma_i \rho([E_i])\;.\label{eq:clgrp2}\end{align}

  Assume now \([D]\in\Cl(X)^\tors\).
  Then \(\rho([D]) = 0\) and we conclude by \eqref{eq:clgrp2} that \(v_i(f) = -r_i a_i\) for all \(i\) and, in particular, \(r_i\mid v_i(f)\).
  Hence, \(f\in \C[V]^H\) by Lemma~\ref{lem:clgrp2} and therefore we can identify \([D'] = [\phi_\ast D]\), or more precisely \(\chi_{[\phi_\ast D]}\), with an element of \(\Hom(G/H, \C^\times)\).
  This means that we obtain a well\-/defined map \[\psi: \Cl(X)^\tors \to \Hom(G/H, \C^\times),~[D]\mapsto[\phi_\ast D]\] by restricting \(\phi_\ast\) to \(\Cl(X)^\tors\).

  We now prove that \(\psi\) is bijective.
  For injectivity, notice that the morphism of groups \[\theta: \Cl(X)\to \Cl(V/G)\oplus \Cl(X)^\mathrm{free},\ [D]\mapsto ([\phi_\ast D], \rho([D]))\] is injective.
  This follows from the exactness of the sequence in Proposition~\ref{prop:ses} noticing that the group \(\bigoplus_{i = 1}^m\Z E_i\) embeds into \(\Cl(X)^\mathrm{free}\), see also \cite[Lemma~4.1.4]{Gra19}.
  The injectivity of \(\theta\) implies the injectivity of \(\psi\): if we have \(\psi([D]) = \psi([D'])\) for \([D], [D']\in\Cl(X)^\tors\), then \(\theta([D]) = \theta([D'])\) as by construction \(\rho([D]) = 0 = \rho([D'])\).

  Now let \(\chi\in \Hom(G/H, \C^\times)\) be a character, which we identify with a class of divisors \([D]\in \Cl(V/G)\).
  By Lemma~\ref{lem:clgrp3}, there exists \(0\neq f\in \C[V]_\chi^{[G,G]}\) and we may assume without loss of generality that \(D\in \Div(V/G)\) is effective and \(f\) is the canonical section of \(D\) as in Proposition~\ref{prop:Ddivcorr}.
  By the assumption on \(\chi\), we have \(\frac{1}{r_i}v_i(f) \in \Z\) for all \(i\) by Lemma~\ref{lem:clgrp2}.
  Let \[E\coloneqq -\sum_{i = 1}^m\frac{1}{r_i}v_i(f)E_i\in\Div(X)\] and set \(D' \coloneqq \overline D - E\), where \(\overline D \coloneqq \phi_\ast^{-1}(D)\) is the strict transform of \(D\) via \(\phi\).
  By Proposition~\ref{prop:ses}, we have \([E]\in \ker(\phi_\ast)\) and therefore \([\phi_\ast D'] = [\phi_\ast \overline D] = [D]\).
  Using Proposition~\ref{prop:divfree}, we have \(\rho([\overline D]) = \rho([E])\), hence \(\rho([D']) = 0\) and \([D']\in\Cl(X)^\tors\).
  We conclude \(\psi([D']) = [D]\) and \(\psi\) is surjective.
\end{proof}

Combining Theorem~\ref{thm:mckay}, see Remark~\ref{rem:free}, and Theorem~\ref{thm:clgrp} enables us to describe the class group of \(X\) in general.
\begin{corollary}
  \label{cor:clgrpfull}
  Let \(G\leq \SL(V)\) be a finite group and let \(H\leq G\) be the subgroup generated by the junior elements contained in \(G\).
  Let \(\phi: X\to V/G\) be a \(\Q\)\=/factorial terminalization of \(V/G\).
  Then we have \[\Cl(X) \cong \Z^m \oplus \Ab(G/H)^\vee,\] where \(m\) is the number of junior conjugacy classes in \(G\).
  Further, the canonical embedding \(\iota : \bigoplus_{i = 1}^m\Z E_i \to \Cl(X)^\free\) satisfies \(\coker(\iota) = \overline H^\vee\) with \(\overline H \coloneqq H/(H\cap [G,G])\) as above.
\end{corollary}
\begin{proof}
  The first part follows directly from the mentioned theorems.
  For the second part, we combine Proposition~\ref{prop:ses} and the first part to obtain \(\Ab(G)^\vee \cong \coker(\iota) \oplus \Ab(G/H)^\vee\) and then the claim follows by Lemma~\ref{lem:hbar}.
\end{proof}

\begin{remark}
  As the isomorphism in Theorem~\ref{thm:clgrp} is induced by \(\phi_\ast\), we can see the sequence in Proposition~\ref{prop:ses} as the direct sum of the short exact sequences \[\begin{tikzcd} 0 \arrow{r} & \bigoplus_{i = 1}^m\Z E_i \arrow{r} & \Cl(X)^\free \arrow{r} & \overline H^\vee \arrow{r} & 0\end{tikzcd}\] and \[\begin{tikzcd} 0 \arrow{r} & 0 \arrow{r} & \Cl(X)^\tors \arrow{r} & \Ab(G/H)^\vee \arrow{r} & 0\;.\end{tikzcd}\]
\end{remark}

We obtain \cite[Proposition~4.14]{Yam18} as a further corollary.
\begin{corollary}[Yamagishi]
  \label{cor:clgrp}
  Let \(G\leq \SL(V)\) be a finite group and let \(\phi:X\to V/G\) be a \(\Q\)\=/factorial terminalization of \(V/G\).
  Then the class group \(\Cl(X)\) is free if and only if \(G\) is generated by the junior elements contained in \(G\) together with \([G,G]\).
\end{corollary}

\section{Examples and closing remarks}
\label{sec:ex}

\begin{remark}
  \label{rem:clgrp}
  Note that in Corollary \ref{cor:clgrp} we cannot drop the part `together with \([G,G]\)' for the equivalence, that is, there are groups which are not generated by junior elements such that \(\Cl(X)\) is free.
  For example, let \(\mathsf I\leq \SL_2(\C)\) be the binary icosahedral group \cite[Theorem~5.14]{LT09} and set \(G \coloneqq \{\diag(g, g) \mid g\in \mathsf I\}\leq\SL_4(\C)\).
  The abelianization \(\Ab(\mathsf I) = \{1\}\) is trivial, so the same is true for \(\Ab(G)\).
  However, every non\-/trivial element in \(\mathsf I\) is of age 1, hence all non\-/trivial elements of \(G\) are of age 2 and \(G\) does not contain any junior elements.
  Hence, the class group of a \(\Q\)\=/factorial terminalization of \(\C^4/G\) is trivial and therefore free.
  For an example of a non\-/trivially free class group, one considers the direct product of \(G\) with a group generated by junior elements.
\end{remark}

\begin{example}
  \label{ex:reality}
  As a `reality check', let \(G\leq \SL(V)\) be a group which does not contain any junior elements.
  Then \(\age(g) > 1\) for every non\-/trivial \(g\in G\), so \(V/G\) has terminal singularities by \cite[Theorem~3.21]{Kol13}.
  Hence, \(V/G\) is a \(\Q\)\=/factorial terminalization of itself and Corollary~\ref{cor:clgrpfull} gives \(\Cl(V/G) = \Ab(G)^\vee\) as in Theorem~\ref{thm:clgrpVG}.
\end{example}

\begin{example}
  \label{ex:nontriv}
  For a non\-/trivial example, we consider the group \[G \coloneqq \big\langle\!\diag(-1 , -1, -\zeta_3, -\zeta_3^2)\big\rangle\leq\SL_4(\C)\] of order 6, where \(\zeta_3\) is a primitive third root of unity.
  As \(G\) does not contain any reflections, we have \(\Cl(\C^4/G) \cong \Z/6\Z\).

  To determine the age of elements in \(G\), we need to fix a primitive sixth root of unity.
  However, the two possible choices \(-\zeta_3\) and \(-\zeta_3^2\) both result in the same junior elements of \(G\), namely \[g_1 \coloneqq \diag(1, 1, \zeta_3^2, \zeta_3)\text{ and }g_2 \coloneqq \diag(1, 1, \zeta_3, \zeta_3^2)\;.\]
  By the Reid--Tai criterion \cite[Theorem~3.21]{Kol13}, the existence of junior elements in \(G\) implies that \(V/G\) is not terminal.
  As \(G\) is abelian, the conjugacy classes in \(G\) are trivial.
  So, the rank of the free part of the class group \(\Cl(X)\) of a \(\Q\)\=/factorial terminalization \(X\to \C^4/G\) is 2.
  For the torsion part, we determine that \(G/H \cong C_2\) is cyclic of order 2 and we conclude \[\Cl(X) \cong \Z^2 \oplus \Z/2\Z\;.\]

  We write the elements of \(\Cl(X)\) as 3\-/tuples with the first two entries corresponding to the free part and the last entry corresponding to the torsion part.
  Then the push\-/forward morphism \(\Cl(X) \to \Cl(\C^4/G)\) is induced by \[(1, 0, 0) \mapsto g_1,\ (0, 1, 0) \mapsto g_2,\ (0, 0, 1) \mapsto -I_4\;,\] where \(I_4\) denotes the identity matrix.
\end{example}

\appendix
\section{Age revisited}
\label{sec:app}

Recall that for \(g\in G\) the integer \(\age(g)\) depends on a choice of root of unity.
In this appendix, we study this issue in more detail to show that the results in this paper are in fact independent of any choices.

\begin{remark}
  In \cite{IR96}, Ito and Reid avoid making any choices by defining the age not for the group \(G\), but for the set \(\Gamma\coloneqq\Hom(\mu_R, G)\), where \(\mu_R\) is the group of roots of unity of order \(R\) and \(R\) is a common multiple of the orders of the elements of \(G\).
  On \(\Gamma\), the notion of age is independent of any choices.
  Any primitive root of unity \(\zeta\in\C^\times\) of order \(R\) gives a bijection \(\Gamma\to G,\ \phi\mapsto \phi(\zeta)\) and one may endow \(\Gamma\) with a group operation via this map.
  However, for the arguments in this paper we need a notion of age on \(G\); this is quite common, see for example \cite{Rei02}.
\end{remark}

First of all, to be able to speak about junior elements in a uniform way, we introduce the following definition.
Let \(e(G)\) be the exponent of \(G\) and let \(\zeta\in\C^\times\) be a primitive \(e(G)\)\=/th root of unity.
For any \(g\in G\), we have \(g^{e(G)} = \id_V\), so there are \(0\leq a_i' < e(G)\) such that the eigenvalues of \(g\) are given by \(\zeta^{a_i'}\), \(i = 1,\dots, n\), with \(n \coloneqq \dim V\).
If \(r\) is the order of \(g\), we must have \(\frac{e(G)}{r}\mid a_i'\) and set \(a_i \coloneqq \frac{ra_i'}{e(G)}\in \Z\).

\begin{definition}
  With the above notation, we set \(\age_\zeta(g) \coloneqq \frac{1}{r}\sum_{i = 1}^na_i\).
  We call the integers \(a_1,\dots, a_n\) the \emph{weights} of \(g\) with respect to \(\zeta\).
  We call \(g\) a \emph{\(\zeta\)\=/junior element}, if \(\age_\zeta(g) = 1\).
\end{definition}

The integer \(\age_\zeta(g)\) coincides with \(\age(g)\) as constructed above for an appropriate choice of \(\zeta\).

\begin{lemma}
  \label{lem:app1}
  Let \(\zeta,\eta\in\C^\times\) be primitive \(e(G)\)\=/th roots of unity.
  Then there is a bijection (of sets) \(\phi:G\to G\) such that \(\age_{\eta}(\phi(g)) = \age_{\zeta}(g)\) for all \(g\in G\).
  Further, the weights of \(g\in G\) with respect to \(\zeta\) and \(\phi(g)\) with respect to \(\eta\) coincide.
\end{lemma}
\begin{proof}
  By assumption, there is \(a\in \Z_{> 0}\) with \(\eta = \zeta^a\) and \(\gcd(a, e(G)) = 1\).
  Then there is \(b\in \Z_{>0}\) with \(ab\equiv 1\) mod \(e(G)\).
  Hence we have a map \[\phi : G\to G,\ g\mapsto g^a\] with inverse \(g\mapsto g^b\).

  Let \(g\in G\) be an element of order \(r\in\Z_{>0}\).
  There are \(0\leq a_i < r\) such that the eigenvalues of \(g\) are given by \(\zeta^{\frac{e(G)a_i}{r}}\).
  Then the eigenvalues of \(\phi(g)\) are given by \((\zeta^{\frac{e(G)a_i}{r}})^a = \eta^{\frac{e(G)a_i}{r}}\) and \[\age_\eta(\phi(g)) = \frac{1}{r} \sum_{i = 1}^na_i = \age_\zeta(g)\] as claimed.
\end{proof}

\begin{lemma}
  \label{lem:app2}
  Let \(\zeta,\eta\in\C^\times\) be primitive \(e(G)\)\=/th roots of unity and write \(H_\zeta\leq G\), respectively \(H_\eta\leq G\), for the subgroup of \(G\) generated by the \(\zeta\)\=/junior elements, respectively the \(\eta\)\=/junior elements.
  Then we have \(H_\zeta = H_\eta\).
\end{lemma}
\begin{proof}
  Let \(\phi:G\to G\) be the bijection in Lemma~\ref{lem:app1}, so taking powers by some \(a\in\Z_{>0}\).
  If \(g\in H_\zeta\) is a \(\zeta\)\=/junior element, then \(\phi(g) = g^a \in H_\eta\) is an \(\eta\)\=/junior element.
  Further, we clearly have \(g^a\in H_\zeta\), hence \(H_\eta \subseteq H_\zeta\) and an analogous argument gives the reverse inclusion.
\end{proof}

\emergencystretch=1ex
\printbibliography

\end{document}